\documentclass[12pt]{article}
\usepackage{fullpage}
\usepackage{epsf,epsfig,amsfonts,amsgen,amsmath,amstext,amsbsy,amsopn,amsthm,amssymb, epstopdf,tikz,mathrsfs}
\usepackage{ebezier,eepic,mathtools,dsfont}
\usepackage{color,colordvi}
\usepackage{multirow}
\usepackage{url}
\usepackage{graphicx}
\usepackage{epsf}
\usepackage[format=hang, margin=10pt]{caption}
\newtheorem{theorem}{Theorem}[section]

\newtheorem{conjecture}[theorem]{Conjecture}
\newtheorem{lemma}[theorem]{Lemma}

\theoremstyle{definition}

\newtheorem{remark}[theorem]{Remark}

\begin{document}

\title{On derivatives and higher-order derivatives of chromatic
polynomials}

\author{Bo Ning\thanks{College of Computer Science, Nankai University, Tianjin 300350, P.R. China. Email: bo.ning@nankai.edu.cn.
Research supported in part by NSFC (No. 12371350) and the Fundamental Research Funds for the Central Universities,
Nankai University.}\,
\qquad
Yan Yang\thanks{School of Mathematics and KL-AAGDM, Tianjin University, Tianjin 300354, P.R. China. Email: yanyang@tju.edu.cn. Corresponding author.}\,
}

\date{}

\maketitle

\begin{abstract}
Let \( G \) be a graph of order \( n \) with maximum degree $\Delta$, and let $P(G,x)$ denote its chromatic polynomial.
We investigate several properties of $P(G,x)$ related to its derivatives and higher-order derivatives.
First, we study the monotonicity of $P(G,x)/x^n$. Dong proved that $(x-1)^nP(G,x)\geq x^nP(G,x-1)$
for all real $x\geq n$. In particular, taking $x=n$ establishes the
Bartels-Welsh ``shameful conjecture" that $P(G,n)/P(G,n-1)>e$.
Fadnavis later showed that the same inequality holds for all real $x\geq 36\Delta^{3/2}$.
We improve this bound by proving that it also holds for all real $x\geq 10\Delta^{3/2}$.
We then consider a conjecture of Dong, Ge, Gong, Ning, Ouyang, and Tay asserting that
\(
\frac{d^k}{dx^k} \bigl( \ln[(-1)^n P(G, x)] \bigr) < 0
\)
for all \( k \geq 2 \) and \( x \in (-\infty, 0) \). We
establish this conjecture for all \( k \geq 2 \) and \( x\leq -3.01\Delta k \).


\end{abstract}

\section{Introduction}
All graphs considered in this paper are simple. For a graph $G=(V, E)$ and a positive integer $x$,  a {\it proper $x$-coloring} of $G$ is a mapping $c: V\rightarrow \{1,\ldots,x\}$ such that
$c(u)\neq c(v)$ whenever $uv\in E$. The {\it chromatic number} $\chi(G)$ is the least $x$ for which $G$ admits a proper $x$-coloring.
Introduced by Birkhoff \cite{B1912} in 1912,
the {\it chromatic polynomial} $P(G,x)$ counts the number of distinct proper $x$-colorings of $G$, i.e.,
$$P(G,x)=|\{c: V\rightarrow \{1,\ldots,x\} \mid c\text{ is proper}\}|.$$
It was once hoped that chromatic polynomials might help resolve the four-color conjecture
because they are closely connected to chromatic numbers: $\chi(G)$ is the smallest positive integer such that $P(G, \chi(G))\neq 0$.
Although this approach to the four-color problem has not been successful so far, chromatic polynomials have attracted widespread attention and remain an active
area of research. For a book devoted to the topic, see Dong, Koh and Teo~\cite{DKT05}.

By the principle of inclusion and exclusion, we have
\begin{equation*}
P(G,x) = \sum_{E' \subseteq E} (-1)^{|E'|}x^{\kappa(E')},
\end{equation*}
where \( \kappa(E') \) denotes the number of connected components in \( E' \).
In 1932, Whitney \cite{W1932} interpreted the coefficients of \( P(G, x) \)
by introducing the notion of broken cycles. Let \( \eta: E \to \{1, 2, \dots, |E|\} \) be a bijection.
For any cycle \( C \) in \( G \), the path \( C - e \) is called a {\it broken cycle} of \( G \)
if \( e \) is the edge on \( C \) with \( \eta(e) \leq \eta(e') \) for every edge \( e' \) on \( C \).
If \( G \) is a graph of order \( n \), then
\[
P(G, x) = \sum_{i=1}^{n} (-1)^{n-i} a_i(G) x^i,
\]
where \( a_i(G) \) is the number of spanning subgraphs of \( G \) with \( n - i \) edges which do not contain broken cycles.

So it is clear that $P(G, x)$ is a polynomial in $x$ of degree $|V|$.
Let $G$ be a graph of order $n$. By the fundamental theorem of algebra, the equation $P(G,x)=0$
has exactly $n$ roots in the complex plane, counted with multiplicity. So we can write
$$P(G,x)=\prod\limits_{i=1}^{n}(x-\alpha_i)$$ where $\alpha_1,\ldots, \alpha_n$ are roots of $P(G,x)=0$.

Let $\Delta(G)$ (or simply $\Delta$) be the maximum degree of a graph $G$ and
$$\rho(G)=\max\{|\alpha|: P(G,\alpha)=0\}$$ be the maximum modulus of chromatic roots of $G$.
Sokal \cite{S01} proved the following results concerning a bound on $\rho(G)$.

\begin{theorem}[\cite{S01}]\label{Thm:1}
 There exists a constant $K$ such that for every graph $G$, $P(G,x) \neq 0$ for all $x \in \mathbb{C}$ with $|x| > K \Delta$.
\end{theorem}
Sokal \cite{S01} proved that one may take $K\leq 7.97$. Fern\'{a}ndez and Procacci \cite{FP08} improved
this to $K\leq 6.91$, and
Jenssen, Patel, and Regts \cite{J24} further improved it to $K \leq 5.94$. More recently, Bencs and
Regts \cite{BR26} established $K \leq 4.25$.
Currently, the best bound on the maximum modulus of chromatic roots is $\rho(G)\leq 4.25\Delta$.

In this paper, we study several properties of $P(G,x)$ arising from
derivatives and higher-order derivatives of related functions. All the results obtained are related to $\rho(G)$.

We first focus on the monotonicity of $P(G,x)/x^n$ for $x\geq \rho(G) $ where $G$ is a graph of order $n$. In 1995, Bartels and Welsh \cite{B1995} conjectured that for any $n$-vertex graph $G$, $$\frac{P(G,n)}{P(G,n-1)}\geq \frac{n^n}{(n-1)^n}>e=2.7182818...$$ holds when they studied the {\it  mean color number} $\mu(G)$ of a graph $G$, that is, the average of number of colors used in all $n$-colorings of $G$. The above inequality is equivalent to $\mu(G)\geq \mu(O_n)$ where $G$ is an $n$-vertex graph and $O_n$ is the empty graph on $n$ vertices. Because the conjecture $\mu(G)\geq \mu(O_n)$ seems intuitively obvious but is hard to prove, it is called ``the shameful conjecture''.
Seymour \cite{S97} proved that for any $n$-vertex graph $G$,
$$\frac{P(G,n)}{P(G,n-1)}\geq\frac{685}{252}= 2.7182539...$$ which came very close to resolving the conjecture.
Ultimately, Dong \cite{D2000} proved that the shameful conjecture is true
by an induction argument.

\begin{theorem}[\cite{D2000}] For every graph of order $n$ and real number $x\geq n$,
\begin{equation}\label{e:D2000}\frac{P(G,x)}{P(G,x-1)}\geq \frac{x^n}{(x-1)^n}.\end{equation} In particular, taking $x=n$, we obtain
$$\frac{P(G,n)}{P(G,n-1)}\geq \frac{n^n}{(n-1)^n}>e.$$
\end{theorem}

In 2015, Fadnavis \cite{F15} showed that the inequality \eqref{e:D2000} holds for all real $x\geq 36\Delta^{3/2}$, i.e.,  $P(G,x)/x^n$ is increasing for all $x\geq 36\Delta^{3/2}$.
In this paper, we improve Fadnavis's results by showing the following theorem.

\begin{theorem}\label{Thm:Main1}If $G$ is a graph of order $n$ and maximum degree $\Delta$, then
$$\frac{P(G,x)}{P(G,x-1)}\geq \frac{x^n}{(x-1)^{n}}$$ for all real $x\geq10\Delta^{3/2}.$
\end{theorem}

Our second main result concerns the higher-order derivatives of $\ln[(-1)^n P(G, x)]$ with $x<0$.
Let \( P^{(i)}(G, x) \) denote the \( i \)-th derivative of \( P(G, x) \). Recently,
Bernardi and Nadeau \cite{BN20} gave a nice combinatorial interpretation of \( P^{(i)}(G, -j) \)
for any nonnegative integers \( i \) and \( j \) in terms of acyclic orientations.
When \( i = 0 \) and \( j = 0 \),  their result recovers the classical interpretations
due to Stanley \cite{S1973} and to Greene and Zaslavsky \cite{GZ83}, respectively.

In 2020, Dong, Ge, Gong, Ning, Ouyang, and Tay \cite{D21} studied the parameter
$$\epsilon(G)=\sum\limits_{i=1}^{n} (n-i) a_i(G)/\sum\limits_{i=1}^{n} a_i(G)$$
namely the mean size of a broken-cycle-free spanning subgraph of a graph
$G$ of order $n$, where \( a_i(G) \) denotes the number of spanning subgraphs of \( G \)
with \( n - i \) edges that do not contain broken cycles. They proved that
$\epsilon(T_n)<\epsilon(G)<\epsilon(K_n)$ for every connected graph $G$ of order $n$ that is neither the
complete graph $K_n$ nor a tree $T_n$. As a consequence, they confirmed a
conjecture due to Lundow and Markstr\"{o}m \cite{LM06}.
In their proof, they introduced the function $\epsilon(G,x)=P'(G,x)/P(G,x)$, which satisfies
$\epsilon(G)=n+\epsilon(G,-1)$  for any graph $G$ of order $n$. Hence, for any graphs $G$ and $H$ of the same order,
the inequality $\epsilon(G)<\epsilon(H)$ is equivalent to $\epsilon(G,-1)<\epsilon(H,-1)$. In \cite{D21}, the authors proved that for any connected graph $G$ of order $n$ which is neither the complete graph nor a tree, $\epsilon(T_n,x)<\epsilon(G,x)<\epsilon(K_n,x)$ holds for all $x<0$. Then $\epsilon(T_n)<\epsilon(G)<\epsilon(K_n)$ follows.

Clearly, for any graph \( G \) of order \( n \),
\begin{equation*}
\epsilon(G,x)=\frac{P'(G, x)}{P(G, x)}=\frac{d}{dx} (\ln[(-1)^n P(G, x)])< 0
\end{equation*}
holds for all \( x < 0 \). Recall that for $x<0$, \((-1)^n P(G, x)>0\) and \((-1)^n P'(G, x)<0\).
Motivated by this observation, the authors of \cite{D21} conjectured that the same property holds for higher
derivatives of the function \(\ln[(-1)^n P(G, x)]\) for $x<0$.

\begin{conjecture}[\cite{D21}]\label{con:1.4} Let \( G \) be a graph of order \( n \). Then
\[
\frac{d^k}{dx^k} (\ln[(-1)^n P(G, x)]) < 0
\]
holds for all \( k \geq 2 \) and \( x \in (-\infty, 0) \).
\end{conjecture}

In this paper, we prove this conjecture for all \( k \geq 2 \) and \( x\leq -3.01\Delta k \).
\begin{theorem}\label{Thm:Main2}Let \( G \) be a graph of order \( n \). Then
\[
\frac{d^k}{dx^k} (\ln[(-1)^n P(G, x)]) < 0
\]
holds for all \( k \geq 2 \) and $x\leq -3.01\Delta k$.
\end{theorem}
The paper is organized as
follows.
In Section \ref{Sec:2}, we prove Theorem \ref{Thm:Main1} by following the framework of \cite{F15}
but using a sharper coefficient bound.
In Section \ref{Sec:3}, we prove Theorem \ref{Thm:Main2}.

\section{Monotonicity of $P(G,x)/x^n$}\label{Sec:2}
For a graph $G$ of order $n$, we define $$F_{G}(x):=\ln \frac{P(G,x)}{x^n},~~~~~x>\rho(G).$$
The derivative $F'_G(x) > 0$ is equivalent to the function $P(G,x)/x^n$ being monotonically increasing.
In \cite{F15}, Fadnavis gave the Laurent expansion of $F_G(x)$.
\begin{theorem} [\cite{F15}]\label{Thm:2.1}
For a graph $G$ of order $n$ and $x>\rho(G)$, we have
\[
F_{G}(x)= \sum_{i=1}^\infty c_i x^{-i},
\]
where $c_i=-\dfrac{1}{i}\sum_{j=1}^n \alpha_j^i$, and $\alpha_1,\ldots, \alpha_n$ are the roots of $P(G,x)=0$.
\end{theorem}

Fadnavis \cite{F15} also deduced the exact values of $c_1$ and $c_2$.
\begin{lemma} [\cite{F15}] \label{Lem:2.2}
Let $|T|$ be the number of triangles in $G$.
Then $c_1=-|E|$ and $c_2=-|T|-|E|/2.$
\end{lemma}
For the remaining coefficients $c_i$,  by Theorem \ref{Thm:1}, we have
\begin{equation}\label{e4}
|c_i| \leq \frac{1}{i} \sum\limits_{j=1}^n |\alpha_j|^i \leq \frac{n}{i}(\rho(G))^i.
\end{equation}

These yield the following intermediate criterion.

\begin{lemma}\label{le0} Let $G$ be a graph with $n$ vertices, $m$ edges, $t$ triangles, and $\rho(G)$ be
the maximum modulus of the roots of  $P(G,x)$. Then, for all real $x > \rho(G)$,
\begin{eqnarray}\label{e0}
F'_G(x) \geq \frac{m}{x^2} + \frac{2t + m}{x^3} - \frac{n(\rho(G))^3}{x^3(x - \rho(G))}.\end{eqnarray}
Consequently, if
\begin{eqnarray}\label{e1}
(mx + 2t + m)(x - \rho(G)) > n(\rho(G))^3, \end{eqnarray}
then $F'_G(x) > 0$.
\end{lemma}

\begin{proof}By differentiating the Laurent expansion of $F_G$ in Theorem \ref{Thm:2.1}, and using Lemma \ref{Lem:2.2} and inequality \eqref{e4}, we get
\begin{align*}
F'_G(x) &= \frac{m}{x^2} + \frac{2t + m}{x^3} - \sum_{i \geq 3} i c_i x^{-i-1} \\
&\geq \frac{m}{x^2} + \frac{2t + m}{x^3} - \sum_{i \geq 3} n (\rho(G))^i x^{-i-1} \\
&= \frac{m}{x^2} + \frac{2t + m}{x^3} - \frac{n (\rho(G))^3}{x^4} \sum_{j \geq 0} \left( \frac{\rho(G)}{x} \right)^j \\
&= \frac{m}{x^2} + \frac{2t + m}{x^3} - \frac{n (\rho(G))^3}{x^3 (x - \rho(G))}.
\end{align*}
This proves \eqref{e0}. If \eqref{e1} holds, then the right-hand side of \eqref{e0} is positive, so $F'_G(x) > 0$.
\end{proof}

\begin{lemma}\label{le1}Let $G$ be a connected graph that is not a tree. Then $F'_G(x) > 0$ whenever
\begin{eqnarray}\label{e2}
(x + 1)(x - \rho(G)) > (\rho(G))^3. \end{eqnarray}
Furthermore, $F_G$ is monotonically increasing on
$(x_0(G), \infty),$  where
$$x_0(G) := \frac{\rho(G) - 1 + \sqrt{(\rho(G) + 1)^2 + 4(\rho(G))^3}}{2}.$$

\end{lemma}

\begin{proof}If $G$ is a connected graph with $n$ vertices, $m$ edges, $t$ triangles and is not a tree, then $m \geq n$ and $t \geq 0$. Hence
\[
(mx + 2t + m)(x - \rho(G)) \geq m(x + 1)(x - \rho(G)) \geq n(x + 1)(x - \rho(G)).
\]
Therefore, inequality \eqref{e2} implies inequality \eqref{e1}.  From Lemma \ref{le0},  $F'_G(x) > 0$ follows. And one can compute that the explicit expression for $x_0(G)$ is just the larger root of the equation $(x + 1)(x - \rho(G)) = (\rho(G))^3$. The proof is complete.
\end{proof}

Now we are ready to prove Theorem \ref{Thm:Main1}.

\begin{proof}[\bf{Proof of Theorem \ref{Thm:Main1}}] It suffices to consider connected graphs, because if the theorem holds for
every connected component of a disconnected graph, then the theorem also holds for this disconnected graph due to the multiplicativity of the chromatic polynomial.

If $G$ is a tree, then
$$P(G,x)=x(x-1)^{n-1},  ~~~\mbox{and}~~~\frac{P(G,x)}{x^n}=(1-\frac{1}{x})^{n-1}.$$
Note that $P(G,x)/x^n$ is increasing for $x>1$. So we may assume $G$ is not a tree.

We will discuss the following cases according to the maximum degree $\Delta$ of $G$.

\noindent{\bf Case 1:} $\Delta=2$.

In this case, $G$ must be a cycle whose chromatic polynomial is
$$P(G,x)=(x-1)^n+(-1)^n(x-1).$$ We now determine its
chromatic roots and hence obtain the maximum modulus of the roots.
Let $(x-1)^n+(-1)^n(x-1)=0$. Clearly, $x=1$ is a root. Then we set \begin{eqnarray}\label{e3}(x-1)^{n-1}=(-1)^{n-1}\end{eqnarray} to solve for the other roots.
Let $x-1=re^{i\theta}$, because $-1=e^{i\pi}$, the equation \eqref{e3} reduces to
\begin{eqnarray*}r^{n-1}e^{i(n-1)\theta}=e^{i(n-1)\pi},
\end{eqnarray*} which yields the two equations
\begin{eqnarray*}r^{n-1}=1,~~~\mbox{and}~~~ (n-1)\theta=(n-1)\pi+2k\pi,~~~k=0,1,\ldots,n-2.
\end{eqnarray*} Solving these equations, we obtain $r=1$, $\theta=\frac{(n-1)\pi+2k\pi}{n-1}$. Hence the roots of  equation \eqref{e3} are
\begin{eqnarray*}x=1+\cos\theta+i\sin\theta=1+e^{i\frac{(n-1)\pi+2k\pi}{n-1}}, ~~~k=0,1,\ldots,n-2.
\end{eqnarray*}
One can compute that the moduli of these roots are
$$|x|=\sqrt{(1+\cos\theta)^2+\sin^2\theta}=2|\sin\frac{k\pi}{n-1}|, ~~~k=0,1,\ldots,n-2.$$
Then we can check that the maximum moduli of the roots $\rho(G)\leq 2$,
with equality if and only if $G$ is an odd cycle (i.e., when $k=\frac{n-1}{2}$).

By Lemma \ref{le1}, when $G$ is a cycle, $F_{G}(x)$ is increasing
whenever $$x>\frac{2-1+\sqrt{3^2+4\times8}}{2}=\frac{1+\sqrt{41}}{2}=3.70156....$$
Moreover, when $x\geq10\Delta^{3/2}$ with $\Delta=2$, the inequality $x> 3.71$ is trivially satisfied.

\noindent{\bf Case 2:} $\Delta\geq 3$.

In this case, $\rho(G)\leq 4.25\Delta$. By Lemma \ref{le1},
it suffices to find a
constant C such that  $C\Delta^{3/2}\geq x_0(G)$, where $x_0(G)$ is defined in Lemma \ref{le1}.
Indeed, we seek $C$ for which
$$( C \Delta^{3/2} + 1 )( C \Delta^{3/2} - \frac{17}{4} \Delta ) > (\frac{17}{4} \Delta)^3.$$
We define $$H(y):=( C y^{3/2} + 1 )( C y^{3/2} - \frac{17}{4}y )- (\frac{17}{4} y)^3, ~~~y\geq 3.$$
Then, $$H'(y)=3(C^2 - \frac{4913}{64})y^2 -\frac{85}{8}C(\sqrt{y})^3+\frac{3}{2}C\sqrt{y} - \frac{17}{4}.$$
When $C\geq 9$, we have $C^2 - 4913/64>0$ and $\frac{3}{2}C\sqrt{y}> \frac{17}{4}$.
Furthermore, we compute that if $C> 9.85$, then
\begin{eqnarray}\label{e5}3(C^2 - \frac{4913}{64})\sqrt{3} -\frac{85}{8}C> 0.\end{eqnarray}
When $C>9.85$ and $y\geq 3$, we have
$$(\sqrt{y})^3\big(3(C^2 - \frac{4913}{64})\sqrt{y} -\frac{85}{8}C\big)> 0.$$
Hence $H'(y)>0$, so $H(y)$ is increasing for $y\geq 3$ whenever $C>9.85$.
This allows us to focus on the smallest value of $C$ for which $H(3)>0$, where
$$H(3)=27C^2-\frac{141\sqrt{3}}{4}C-\frac{133467}{64}.$$
One can compute that when $C\geq 10$, $H(3)>0$. The proof is complete.
\end{proof}

\begin{remark} In the proof of Theorem \ref{Thm:Main1}, we deal with the case $\Delta=2$ separately, aiming to seek a smaller constant
$C$. If we unify the two cases, then $H(y)$ will be defined on $y\geq 2$,  and the inequality \eqref{e5} will become
$$3(C^2 - \frac{4913}{64})\sqrt{2} -\frac{85}{8}C> 0.$$ Consequently, we need a constant $C$
greater than 10 for this inequality to hold.
\end{remark}

\begin{remark} In both our proof and the proof in \cite{F15}, the $\rho(G)$ is needed to bound the absolute value of
$c_i$ as in \eqref{e4}. Thus, a new bound on $\rho(G)$ gives a new range of
$x$ on which $P(G,x)/x^n$ is monotonically increasing. In  \cite{F15}, Fadnavis used the bound $\rho(G)\leq 8\Delta$, which has now been improved to
$\rho(G)\leq 4.25\Delta$. Substituting this new bound into Fadnavis's proof would yield that $P(G,x)/x^n$ is monotonically increasing for $x \geq 14\Delta^{3/2}$.
This result is still weaker than ours, because our proof employs a sharper coefficient bound $c_i\leq \frac{n}{i}(\rho(G))^i$ instead of replacing $n$ by $2|E|$.
\end{remark}

\section{Higher  derivatives of $\ln[(-1)^n P(G, x)]$}\label{Sec:3}
In this section, we first prove that each summand in
equation (\ref{e6}) below is nonpositive for $x\leq -\rho(G)\csc(\frac{\pi}{2k})$. Theorem \ref{Thm:Main2} then follows from an elementary estimate
on $\csc\left(\frac{\pi}{2k}\right)$ and $\rho(G)\leq 4.25\Delta$.

For a graph $G$ of order $n$, we write
$$L_G(x):=\ln[(-1)^n P(G, x)], ~~~~~x<0.$$
Let $\alpha_1,\dots,\alpha_n$ be the roots of $P(G,x)=0$. Then $P(G,x)=\prod_{j=1}^n (x-\alpha_j)$.
For $k\ge 1$ and $x<0$, one computes the $k$-th derivative of $L_G(x)$,
\[
L_G^{(k)}(x)=(-1)^{k-1}(k-1)!\sum_{j=1}^n \frac{1}{(x-\alpha_j)^k}.
\]
By taking real parts, we have
\begin{eqnarray}\label{e6}L_G^{(k)}(x)&=&\Re\Big((-1)^{k-1}(k-1)!\sum_{j=1}^n \frac{1}{(x-\alpha_j)^k}\Big)\nonumber\\
&=&(k-1)!\sum_{j=1}^n \Re\left(\frac{(-1)^{k-1}}{(x-\alpha_j)^k}\right),
\end{eqnarray}
because $L_G^{(k)}(x)$ is real and the real part operator $\Re$ is linear over addition.

The following lemma is the main tool in this section.

\begin{lemma}\label{mlemma} Let $R>0$ and $x<0$ satisfy
\[
|x| \geq R \csc\left(\frac{\pi}{2k}\right).
\]
Then for all integers $k\geq 2$ and $z\in\mathbb{C}$ with $|z|\leq R$,
\[
\Re\left(\frac{(-1)^{k-1}}{(x-z)^k}\right) \leq 0.
\]
Moreover, the inequality is strict when $z=0$.
\end{lemma}

\begin{proof}Set $\phi := \arcsin\left(\frac{R}{|x|}\right)$.
Since $|z|\leq R$, the point $x-z$ lies in the closed disk centered at the negative real number $x$ with radius $R$.
Elementary geometry shows that this disk is contained in the sector
$\{ r e^{i(\pi+\theta)} : r>0,\ |\theta|\leq\phi \}$.
Hence we may write
\[
x-z = |x-z| e^{i(\pi+\theta)}
\]
with $|\theta|\leq\phi$. Therefore
\begin{eqnarray*}\Re(\frac{(-1)^{k-1}}{(x-z)^k})&=& \Re( (-1)^{k-1} |x-z|^{-k} e^{-ik(\pi+\theta)})\\
&=& (-1)^{k-1} |x-z|^{-k}\cos(k(\pi+\theta))\\
&=& - |x-z|^{-k} \cos(k\theta).
\end{eqnarray*}
We set $\phi\leq \frac{\pi}{2k}$, then we have $\cos(k\theta)\geq 0$, and hence $- |x-z|^{-k} \cos(k\theta)\leq 0$ follows.
Because $\phi\leq \frac{\pi}{2k}$ implies
$|x| \geq R \csc\left(\frac{\pi}{2k}\right),$ the lemma follows.
If $z=0$, then $\theta=0$. Therefore, $\Re\left(\frac{(-1)^{k-1}}{(x-z)^k}\right)=-|x|^{-k}<0$.  The proof is complete.
\end{proof}

The range of $x$ in the above lemma can be extended to all nonzero real numbers.

\begin{lemma}\label{elemma}Let $R>0$ and $x\in \mathbb{R}\setminus \{0\}$ satisfy
\[
|x| \geq R \csc\left(\frac{\pi}{2k}\right).
\]
Then for all integers $k\geq 2$ and $z\in\mathbb{C}$ with $|z|\leq R$,
\[
\Re\left(-\frac{\operatorname{sgn}(x)^{k}}{(x-z)^k}\right) \leq 0.
\]
Moreover, the inequality is strict when $z=0$.
\end{lemma}

\begin{proof}Set
\[
\phi := \arcsin\left(\frac{R}{|x|}\right) \leq \frac{\pi}{2k}.
\]
Let
\[
\alpha :=
\begin{cases}
0, & x > 0, \\
\pi, & x < 0.
\end{cases}
\]
Since $|z| \leq R$, the point $x-z$ lies in the closed disk centered at the real number $x$ with radius $R$.
Elementary geometry shows that this disk is contained in the sector
\[
\{ r e^{i(\alpha+\theta)} : r > 0,\ |\theta| \leq \phi \}.
\]
Hence we may write
\[
x - z = |x - z| e^{i(\alpha+\theta)}
\]
with
\[
|\theta| \leq \phi \leq \frac{\pi}{2k}.
\]
Therefore,
\[
\Re\left(-\frac{\operatorname{sgn}(x)^k}{(x-z)^k}\right)
= \Re\left(-\operatorname{sgn}(x)^k |x-z|^{-k} e^{-ik(\alpha+\theta)}\right).
\]
Now, by the definition of $\alpha$, we have
\[
e^{-ik\alpha} = \operatorname{sgn}(x)^k.
\]
Thus
\[
\Re\left(-\frac{\operatorname{sgn}(x)^k}{(x-z)^k}\right)
= \Re\left(-|x-z|^{-k} e^{-ik\theta}\right)
= -|x-z|^{-k} \cos(k\theta) \leq 0,
\]
because $|k\theta| \leq \pi/2$ implies $\cos(k\theta) \geq 0$.

If $z = 0$ then $\theta = 0$. Therefore, $\Re\left(-\frac{\operatorname{sgn}(x)^{k}}{(x-z)^k}\right)=-|x|^{-k} < 0$.
The proof is complete.
\end{proof}

\begin{theorem} \label{t1}For a graph $G$ of order $n$ and $k\geq 2$,
\[L_G^{(k)}(x)<0 \] holds for all $x\leq -4.25\Delta\csc(\frac{\pi}{2k})$.\end{theorem}

\begin{proof} Let $\alpha_1,\dots,\alpha_n$ be the roots of $P(G,x)=0$. From equation \eqref{e6},
$$L_G^{(k)}(x)=(k-1)!\sum_{j=1}^n \Re\left(\frac{(-1)^{k-1}}{(x-\alpha_j)^k}\right), ~~k=1,2,\ldots.$$

If $\Delta=0$, then $G$ is the empty graph on $n$ vertices, $P(G,x)=x^n$. Hence $$L_G^{(k)}(x)=-(k-1)!n|x|^{-k}<0$$ holds for all $x<0$. So we may assume $\Delta\geq 1$.
Since $|\alpha_j|\le \rho(G)$ for every $j$, Lemma \ref{mlemma} shows that each summand in Equation \eqref{e6} is nonpositive for
$x\le -\rho(G)\csc\left(\frac{\pi}{2k}\right)$. Moreover, $0$ is a root of $P(G,x)$. There exists a term $\Re((-1)^{k-1}/x^k)$ in the sum
which is strictly negative by Lemma \ref{mlemma}. Hence $L_G^{(k)}(x)<0$ holds for all $x\le -\rho(G)\csc\left(\frac{\pi}{2k}\right)$.
Furthermore, Bencs and Regts \cite{BR26} proved that $\rho(G)\leq 4.25\Delta$, so the theorem follows.
\end{proof}

By estimating  $\csc(\frac{\pi}{2k})$ from above, we derive Theorem \ref{Thm:Main2}.

\begin{proof}[\bf{Proof of Theorem \ref{Thm:Main2}}] Consider $f(t) := \frac{1}{t} \csc\left(\frac{\pi}{2t}\right), t \geq 2$. With $u = \pi/(2t)$, one can compute that $$f'(t) = \frac{\csc u}{t^2}(u \cot u - 1) < 0,$$ since $u \cot u < 1$ for $u \in (0, \pi/4]$. Thus $f(t)$ is decreasing on $[2, \infty)$.
Furthermore,
\[\csc\left(\frac{\pi}{2k}\right) = kf(k) \leq kf(2) = \frac{k}{\sqrt{2}}<0.708k  \quad (k \geq 2).\]
Then, we have $-4.25\Delta\csc(\frac{\pi}{2k})>-4.25\times 0.708k \Delta> -3.01\Delta k$.
Combining this with Theorem \ref{t1}, the theorem follows.
\end{proof}

\begin{remark}Any improvement on $\rho(G)$ immediately improves the conclusion of Theorems \ref{t1} and \ref{Thm:Main2}. For example,
Bencs and Regts \cite{BR26} proved that if $G$ is claw-free, then $\rho(G) \le 3.81\Delta$. Moreover, if $\Delta \ge 3$  and the girth of $G$ is sufficiently large, then
$\rho(G) \le 3.60\Delta$. Therefore, from Theorems \ref{t1} and \ref{Thm:Main2}, we have
\begin{enumerate}
    \item[(i)] If $G$ is claw-free, then
    \[
    L_G^{(k)}(x) < 0 \quad \text{for} \quad x \le -3.81\Delta\csc\left(\frac{\pi}{2k}\right),
    \]
    and in particular, since $3.81\times 0.708 <2.70$, the same conclusion holds for all
    $x \le -2.70\Delta k.$

    \item[(ii)] If $\Delta \ge 3$ and the girth of $G$ is sufficiently large, then
    \[
    L_G^{(k)}(x) < 0 \quad \text{for} \quad x \le -3.60\Delta\csc\left(\frac{\pi}{2k}\right),
    \]
    and in particular, since $3.60\times 0.708 <2.55$, the same conclusion holds for all
    $x \le -2.55\Delta k.$
\end{enumerate}

\end{remark}

\section*{Acknowledgement}
The authors are supported by National Natural Science Foundation of China (No. 12371350).

\end{document}